\documentclass[a4paper]{article}
\usepackage{amsthm,amssymb,amsmath,enumerate,graphicx}

\newcommand{\comment}[1]{}

\newtheorem{proposition}{Proposition}[section]
\newtheorem{definition}[proposition]{Definition}
\newtheorem{theorem}[proposition]{Theorem}
\newtheorem{corollary}[proposition]{Corollary}
\newtheorem{lemma}[proposition]{Lemma}

\newtheorem{problem}{Problem}

\newcommand{\famun}{\bigcup}

\newcommand{\Debug}{0}

\ifnum \Debug = 1 
\else  \fi

\newcommand{\showFig}[4][0.8\linewidth]{
   \begin{figure}[htbp]
   \centering
   \noindent
   \includegraphics[width=#1]{#2}
   \caption{\small #4}
   \label{#3}
   \end{figure}
}

\newcommand{\N}{\mathbb N}
\newcommand{\R}{\mathbb R}

\newcommand{\oo}{\ensuremath{\omega}}
\newcommand{\OO}{\ensuremath{\Omega}}

\newcommand{\Lr}[1]{Lemma~\ref{#1}}
\newcommand{\Tr}[1]{Theorem~\ref{#1}}
\newcommand{\Sr}[1]{Section~\ref{#1}}

\newcommand{\assign}{
  \mathrel{\mathop{:}}=
}

\newcommand{\COMMENT}[1]{}

\newcommand{\emtext}[1]{\text{\em #1}}
\providecommand{\eps}{\varepsilon}
\providecommand{\NN}{\mathbb{N}}
\providecommand{\RR}{\mathbb{R}}
\providecommand{\CCC}{\mathcal{C}}

\providecommand{\PPP}{\mathcal{P}}

\title{\textbf{Geodetic topological cycles in locally finite graphs}}
\author{Agelos Georgakopoulos \and Philipp Spr\"ussel}
\date{}

\begin{document}

\maketitle

\begin{abstract}
We prove that the topological cycle space $\CCC(G)$ of a locally finite graph $G$ is generated by its geodetic topological circles. We further show that, although the finite cycles of $G$ generate $\CCC(G)$, its finite geodetic cycles need not generate $\CCC(G)$.
\end{abstract}

\section{Introduction}
\label{sec:intro}

A finite cycle $C$ in a graph $G$ is called \emph{geodetic} if, for any two vertices $x,y\in C$, the length of at least one of the two $x$--$y$~arcs on $C$ equals the distance between $x$ and $y$ in $G$. It is easy to prove (see \Sr{subsec:geo:finite}):

\begin{proposition}
\label{pro:fin:sim}
The cycle space of a finite graph is generated by its geodetic cycles.
\end{proposition}

Our aim is to generalise Proposition~\ref{pro:fin:sim} to the topological 
cycle space of locally finite infinite graphs. 
\medskip

The \emph{topological cycle space} $\CCC(G)$ of a locally finite graph $G$ was introduced by Diestel and K\"uhn \cite{cyclesI,cyclesII}. It is built not just from finite cycles, but also from infinite \emph{circles}: homeomorphic images of the unit circle $S^1$ in the topological space $|G|$ consisting of $G$, seen as a 1-complex, together with its ends. (See \Sr{sec:notation} for precise definitions.) 
This space $\CCC(G)$  has been shown \cite{locFinTutte,duality,partition,locFinMacLane,cyclesI,hotchpotch} to be the appropriate notion of the cycle space for a locally finite graph: it allows generalisations to locally finite graphs of most of the well-known theorems about the cycle space of finite graphs, theorems which fail for infinite graphs if the usual finitary notion of the cycle space is applied. It thus seems that the topological cycle space is an important object that merits further investigation. (See \cite{cyclesIntro,diestelBook05} for introductions to the subject.)

As in the finite case, one fundamental question is which natural subsets of the topological cycle space generate it, and how. It has been shown, for example, that the fundamental circuits of topological spanning trees do (but not those of arbitrary spanning trees) \cite{cyclesI}, or the non-separating induced cycles \cite{locFinTutte}, or that every element of $\CCC(G)$ is a sum of disjoint circuits \cite{cyclesII, diestelBook05, vvv}---a trivial observation in the finite case, which becomes rather more difficult for infinite~$G$. (A shorter proof, though still non-trivial, is given in \cite{hotchpotch}.) Another standard generating set for the cycle space of a finite graph is the set of geodetic cycles (Proposition~\ref{pro:fin:sim}), and it is natural to ask whether these still generate $\CCC(G)$  when $G$ is infinite.

But what is a geodetic topological circle? One way to define it would be to apply the standard definition, stated above before Proposition~\ref{pro:fin:sim}, to arbitrary circles, taking as the length of an arc the number of its edges (which may now be infinite). As we shall see, Proposition~\ref{pro:fin:sim}  will fail with this definition, even for locally finite graphs. Indeed, with hindsight we can see why it should fail: 
when $G$ is infinite then giving every edge length~$1$ will result in path lengths that distort rather than reflect the natural geometry of $|G|$: edges `closer to' ends must be shorter, if only to give paths between ends finite lengths.

It looks, then, as though the question of whether or not Proposition~\ref{pro:fin:sim} generalises might depend on how exactly we choose the edge lengths in our graph. However, our main result is that this is not the case: we shall prove that no matter how we choose the edge lengths, as long as the resulting arc lengths induce a metric compatible with the topology of $|G|$, the geodetic circles in $|G|$ will generate $\CCC(G)$. Note, however, that 
the question of which circles are geodetic does 
depend on our choice of edge lengths, even under the assumption that a metric compatible with the topology of $|G|$ is induced. 

If $\ell:E(G)\to\RR_+$ is an assignment of edge lengths that has the above property, we call the pair $(|G|,\ell)$ a \emph{metric representation of $G$}. We then call a circle $C$ \emph{$\ell$-geodetic} if for any points $x,y$ on $C$ the distance between $x$ and $y$ in $C$ is the same as the distance between $x$ and $y$ in $|G|$. See \Sr{MetRep} for precise definitions and more details.

We can now state the main result of this paper more formally:

\begin{theorem}
\label{thm:generating}
For every metric representation $(|G|,\ell)$ of a connected locally finite graph~$G$, the topological cycle space $\CCC(G)$ of $G$ is 
generated by the $\ell$-geodetic circles in~$G$.
\end{theorem}

Motivated by the current work, the first author initiated a more systematic study of topologies on graphs that can be induced by assigning lengths to the edges of the graph. In this context, it is conjectured that \Tr{thm:generating} generalises to arbitrary compact metric spaces if the notion of the topological cycle space is replaced by an analogous homology \cite{ltop}.

We prove \Tr{thm:generating} in \Sr{sec:proof}, after giving the required definitions and basic facts in \Sr{sec:notation} and showing that Proposition~\ref{pro:fin:sim} holds for finite graphs but not for infinite ones in \Sr{sec:geo}. Finally, in \Sr{sec:further} we will discuss some further problems.

\section{Definitions and background}
\label{sec:notation}

\subsection{The topological space $|G|$ and $\CCC(G)$}\label{subsec:notation:infinite}

Unless otherwise stated, we will be using the terminology of \cite{diestelBook05} for graph-theoretical concepts and that of \cite{armstrong} for topological ones. Let $G=(V,E)$ be a \emph{locally finite} graph --- i.e.\ every vertex has a finite degree --- finite or infinite, fixed throughout this section.

The graph-theoretical distance between two vertices $x,y\in V$, is the minimum $n\in \N$ such that there is an $x$--$y$~path in $G$ comprising $n$ edges. Unlike the frequently used convention, we will not use the notation $d(x,y)$ to denote the graph-theoretical distance, as we use it to denote the distance with respect to a metric $d$ on $|G|$.

A $1$-way infinite path is called a \emph{ray}, a $2$-way infinite path is a \emph{double ray}. A \emph{tail of} a ray $R$ is an infinite subpath of $R$. Two rays $R,L$ in $G$ are \emph{equivalent} if no finite set of vertices separates them. The corresponding equivalence classes of rays are the \emph{ends} of $G$. We denote the set of ends of $G$ by $\Omega = \Omega(G)$, and we define $\hat{V}\assign V\cup \OO$.

Let $G$ bear the topology of a 1-complex, where the 1-cells are real intervals of arbitrary lengths\footnote{Every edge is homeomorphic to a real closed bounded interval, the basic open sets around an inner point being just the open intervals on the edge. The basic open neighbourhoods of a vertex $x$ are the unions of half-open intervals $[x, z)$, one from every edge $[x, y]$ at $x$. Note that the topology does not depend on the lengths of the intervals homeomorphic to edges.}. To extend this topology to \OO, let us define for each end $\oo \in \OO$ a basis of open neighbourhoods. Given any finite set $S \subset V$, let $C = C(S, \oo)$ denote the component of $G - S$ that contains some (and hence a tail of every) ray in $\oo$, and let $\OO(S,\oo)$ denote the set of all ends of $G$ with a ray in $C(S, \oo)$. As our basis of open neighbourhoods of $\oo$ we now take all sets of the form
\begin{equation}
C(S, \oo) \cup \OO(S,\oo) \cup E'(S,\oo)
\label{eq:open}
\end{equation}
where $S$ ranges over the finite subsets of $V$ and $E'(S,\oo)$ is any union of half-edges $(z, y]$, one for every $S$--$C$ edge $e = xy$ of $G$, with $z$ an inner point of $e$. Let $|G|$ denote the topological space of $G \cup \OO$ endowed with the topology generated by the open sets of the form~\eqref{eq:open} together with those of the 1-complex $G$. It can be proved (see \cite{ends}) that in fact $|G|$ is the Freudenthal compactification \cite{Freudenthal31} of the 1-complex $G$.

A continuous map $\sigma$ from the real unit interval $[0,1]$ to $|G|$ is a \emph{topological path} in $|G|$; the images under $\sigma$ of $0$ and $1$ are its \emph{endpoints}. A homeomorphic image of the real unit interval in $|G|$ is an \emph{arc} in $|G|$. Any set $\{x\}$ with $x\in|G|$ is also called an arc in $|G|$. A homeomorphic image of $S^1$, the unit circle in $\R^2$, in $|G|$ is a (\emph{topological cycle} or) \emph{circle} in $|G|$. Note that any arc, circle, cycle, path, or image of a topological path is closed in $|G|$, since it is a continuous image of a compact space in a Hausdorff space.

A subset $D$ of $E$ is a \emph{circuit} if there is a circle $C$ in $|G|$ such that $D=\{e \in E \mid e\subseteq C\}$. Call a family $\mathcal F=(D_i)_{i \in I}$ of subsets of $E$ \emph{thin} if no edge lies in $D_i$ for infinitely many indices $i$. Let the \emph{(thin) sum} $\sum \mathcal F$ of this family be the set of all edges that lie in $D_i$ for an odd number of indices $i$, and let the \emph{topological cycle space} ${\mathcal C}(G)$ of $G$ be the set of all sums of thin families of circuits. In order to keep our expressions simple, we will, with a slight abuse, not stricly distinguish circles, paths and arcs from their edge sets.

\subsection{Metric representations} \label{MetRep}

Suppose that the lengths of the $1$-cells (edges) of the locally finite graph $G$ are given by a function $\ell:E(G)\to\RR_+$. Every arc in $|G|$ is either a subinterval of an edge or the closure of a disjoint union of open edges or half-edges (at most two, one at either end), and we define its \emph{length} as the length of this subinterval or as the (finite or infinite) sum of the lengths of these edges and half-edges, respectively. Given two points $x,y \in |G|$, write $d_{\ell}(x,y)$ for the infimum of the lengths of all $x$--$y$~arcs in~$|G|$. It is straightforward to prove:

\begin{proposition}\label{d_metric}
  If for every two points $x,y \in |G|$ there is an $x$-$y$~arc of finite length, then $d_{\ell}$ is a metric on $|G|$.
\end{proposition}

\comment{
\begin{proof}
  We prove that $d_{\ell}$ satisfies the triangle inequality; the other axioms for a metric are trivially satisfied.
  
  Suppose, for contradiction, that there are points $x,y,z\in|G|$ with $d_{\ell}(x,y) > d_{\ell}(x,z) + d_{\ell}(z,y)$. Put $\eps\assign d_{\ell}(x,y) - \big(d_{\ell}(x,z) + d_{\ell}(z,y)\big) > 0$. By definition of $d_{\ell}$ there is an $x$--$z$~arc shorter than $d_{\ell}(x,z) + \eps/2$. Likewise, there is a $z$--$y$~arc shorter than $d_{\ell}(z,y) + \eps/2$. The union of these two arcs contains an $x$--$y$~arc shorter than $d_{\ell}(x,y)$, a contradiction.
\end{proof}
}

This metric $d_{\ell}$ will in general not induce the topology of $|G|$. If it does, we call $(|G|,\ell)$ a \emph{metric representation} of~$G$ (other topologies on a graph that can be induced by edge lengths in a similar way are studied in \cite{ltop}). We then call a circle $C$ in $|G|$ \emph{$\ell$-geodetic} if, for every two points $x,y\in C$, one of the two $x$--$y$~arcs in $C$ has length~$d_{\ell}(x,y)$. If $C$ is $\ell$-geodetic, then we also call its circuit \emph{$\ell$-geodetic}.

Metric representations do exist for every locally finite graph $G$. Indeed, pick a normal spanning tree $T$ of $G$ with root $x\in V(G)$ (its existence is proved in \cite[Theorem 8.2.4]{diestelBook05}), and define the length $\ell(uv)$ of any edge $uv\in E(G)$ as follows. If $uv \in E(T)$ and $v \in xTu$, let $\ell(uv)= 1/2^{|xTu|}$. If $uv \notin E(T)$, let $\ell(uv)=\sum_{e \in uTv} \ell(e)$. It is easy to check that $d_\ell$ is a metric of $|G|$ inducing its topology \cite{diestelESST}.

\subsection{Basic facts}

In this section we give some basic properties of $|G|$ and $\CCC(G)$ that we will need later.

One of the most fundamental properties of $\CCC(G)$ is that:

\begin{lemma}[\cite{cyclesII}]
  \label{pro:cyclespace:circles}
  For any locally finite graph $G$, every element of $\CCC(G)$ is an edge-disjoint sum of circuits.
\end{lemma}

As already mentioned, $|G|$ is a compactification of the $1$-complex $G$:

\begin{lemma}[{\cite[Proposition~8.5.1]{diestelBook05}}]
  \label{pro:compact}
  If $G$ is locally finite and connected, then $|G|$ is a compact Hausdorff space.
\end{lemma}

The next statement follows at once from \Lr{pro:compact}.

\begin{corollary}
  \label{cor:closure}
  If $G$ is locally finite and connected, then the closure in $|G|$ of an infinite set of vertices contains an end.
\end{corollary}

The following basic fact can be found in \cite[p. 208]{elemTop}.

\begin{lemma}
\label{arc}
The image of a topological path with endpoints $x,y$ in a Hausdorff space $X$ contains an arc in $X$ between $x$ and $y$.
\end{lemma}

As a consequence, being linked by an arc is an equivalence relation on $|G|$; a set $Y\subset|G|$ is called \emph{arc-connected} if $Y$ contains an arc between any two points in $Y$. Every arc-connected subspace of $|G|$ is connected.  Conversely, we have:

\begin{lemma}[\cite{tst}]
  \label{pro:connected}
  If $G$ is a locally finite graph, then every closed connected subspace of $|G|$ is arc-connected.
\end{lemma}

The following lemma is a standard tool in infinite graph theory.

\begin{lemma}[K\"onig's Infinity Lemma \cite{InfLemma}]
  \label{lemma:infinity}
  Let $V_0,V_1,\dotsc$ be an infinite sequence of disjoint non-empty finite sets, and let $G$ be a graph on their union. Assume that every vertex $v$ in a set $V_n$ with $n\ge 1$ has a neighbour in $V_{n-1}$. Then $G$ contains a ray $v_0v_1\dotsm$ with $v_n\in V_n$ for all $n$.
\end{lemma}

\section{Generating $\boldsymbol{\CCC(G)}$ by geodetic cycles}
\label{sec:geo}

\subsection{Finite graphs}
\label{subsec:geo:finite}

In this section finite graphs, like infinite ones, are considered as $1$-complexes where the $1$-cells (i.e.\ the edges) are real intervals of arbitrary lengths. Given a metric representation $(|G|,\ell)$ of a finite graph $G$, we can thus define the \emph{length} $\ell(X)$ of a path or cycle $X$ in $G$ by $\ell(X) = \sum_{e\in E(X)} \ell(e)$. Note that, for finite graphs, any assignment of edge lengths yields a metric representation. A cycle $C$ in $G$ is $\ell$-\emph{geodetic}, if for any $x,y \in V(C)$ there is no $x$--$y$~path in $G$ of length strictly less than that of each of the two $x$--$y$~paths on $C$.

The following theorem generalises Proposition~\ref{pro:fin:sim}.

\begin{theorem}
\label{thm:fin:L}
For every finite graph $G$ and every metric representation $(|G|,\ell)$ of $G$, every cycle $C$ of $G$ can be written as a sum of $\ell$-geodetic cycles of length at most $\ell(C)$.
\end{theorem}

\begin{proof}
Suppose that the assertion is false for some $(|G|,\ell)$, and let $D$ be a cycle in $G$ of minimal length among all cycles $C$ that cannot be written as a sum of $\ell$-geodetic cycles of length at most $\ell(C)$. As $D$ is not $\ell$-geodetic, it is easy to see that there is a path $P$ with both endvertices on $D$ but no inner vertex in $D$ that is shorter than the paths $Q_1$, $Q_2$ on $D$ between the endvertices of $P$. Thus $D$ is the sum of the cycles $D_1\assign P \cup Q_1$ and $D_2\assign P \cup Q_2$. As $D_1$ and $D_2$ are shorter than $D$, they are each a sum of $\ell$-geodetic cycles of length less than $\ell(D)$, which implies that $D$ itself is such a sum, a contradiction.
\end{proof}

By letting all edges have length $1$, \Tr{thm:fin:L} implies Proposition~\ref{pro:fin:sim}.

\subsection{Failure in infinite graphs}
\label{subsec:geo:infinite}

As already mentioned, Proposition~\ref{pro:fin:sim} does not naively generalise to locally finite graphs: there are locally finite graphs whose topological cycle space contains a circuit that is not a thin sum of circuits that are geodetic in the traditional sense, i.e.\ when every edge has length $1$. Such a counterexample is given in Figure~\ref{fig:bsp:classic}. The graph $H$ shown there is a subdivision of the \emph{infinite ladder}; the infinite ladder is a union of two rays $R_x=x_1x_2\dotsm$ and $R_y=y_1y_2\dotsm$ plus an edge $x_ny_n$ for every $n\in\NN$, called the \emph{$n$-th rung} of the ladder. By subdividing, for every $n\ge 2$, the $n$-th rung into $2n$ edges, we obtain $H$. For every $n\in\NN$, the (unique) shortest $x_n$--$y_n$~path contains the first rung $e$ and has length $2n-1$. As every circle (finite or infinite) must contain the subdivision of at least one rung, every geodetic circuit contains $e$. On the other hand, Figure~\ref{fig:bsp:classic} shows an element $C$ of $\CCC(H)$ that
contains infinitely many rungs.
As every circle can contain at most two rungs, we need an infinite family of 
geodetic circuits to generate $C$, but since they all have to contain $e$ the family cannot be thin.

The graph $H$ is however not a counterexample to \Tr{thm:generating}, since the constant edge lengths $1$ do not induce a metric of $|H|$.

\showFig{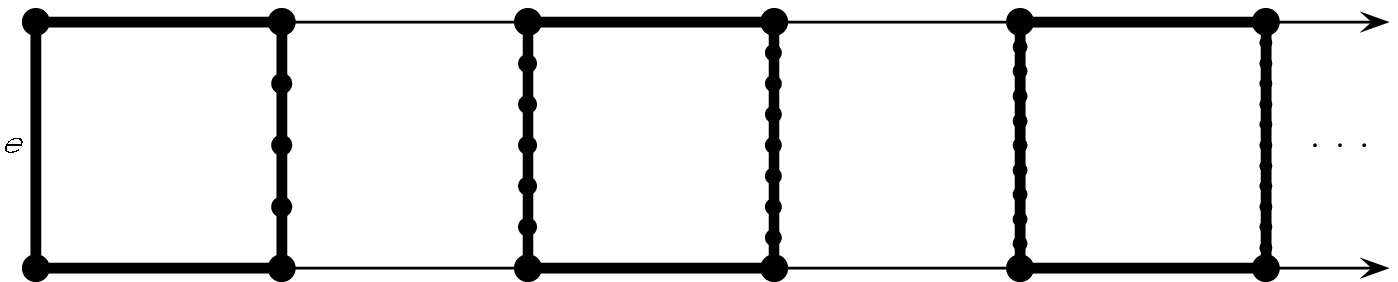}{fig:bsp:classic}{A $1$-ended graph and an element of its topological cycle space (drawn thick) which is not the sum of a thin family of geodetic circuits.}

\section{Generating $\boldsymbol{\CCC(G)}$ by geodetic circles}
\label{sec:proof}

Let $G$ be an arbitrary connected locally finite graph, finite or infinite, consider a fixed metric representation $(|G|,\ell)$ of $G$ and write $d=d_{\ell}$. We want to assign a length to every arc or circle, but also to other objects like elements of $\CCC(G)$. To this end, let $X$ be an arc or circle in $|G|$, an element of $\CCC(G)$, or the image of a topological path in $|G|$. It is easy to see that for every edge $e$, $e \cap X$ is the union of at most two subintervals of $e$ and thus has a natural length which we denote by $\ell(e \cap X)$; moreover, $X$ is the closure in $|G|$ of $\bigcup_{e\in G}(e \cap X)$ (unless $X$ contains less than two points). We can thus define the \emph{length} of $X$ as $\ell(X)\assign\sum_{e\in G}\ell(e \cap X)$. 

Note that not every such $X$ has finite length (see \Sr{sec:further}). But the length of an $\ell$-geodetic circle $C$ is always finite. Indeed, as $|G|$ is compact, there is an upper bound $\eps_0$ such that $d(x,y) \leq \eps_0$ for all $x,y\in |G|$. Therefore, $C$ has length at most $2\eps_0$.

For the proof of Theorem~\ref{thm:generating} it does not suffice to prove that every circuit is a sum of a thin family of $\ell$-geodetic circuits. (Moreover, the proof of the latter statement turns out to be as hard as the proof of Theorem~\ref{thm:generating}.) For although every element $C$ of $\CCC(G)$ is a sum of a thin family of circuits (even of finite circuits, see \cite[Corollary~8.5.9]{diestelBook05}), representations of all the circles in this family as sums of thin families of $\ell$-geodetic circuits will not necessarily combine to a similar representation for $C$, because the union of infinitely many thin families need not be thin.

In order to prove Theorem~\ref{thm:generating}, we will use a sequence $\hat S_i$ of finite auxiliary graphs whose limit is $G$. Given an element $C$ of $\CCC(G)$ that we want to represent as a sum of $\ell$-geodetic circuits, we will for each $i$ consider an element $C | \hat S_i$ of the cycle space of $\hat S_i$ induced by $C$ --- in a way that will be made precise below --- and find a representation of $C | \hat S_i$ as a sum of geodetic cycles of $\hat S_i$, provided by \Tr{thm:fin:L}. We will then use the resulting sequence of representations and compactness to obtain a representation of $C$ as a sum of $\ell$-geodetic circuits.

\subsection{Restricting paths and circles}
\label{subsec:proof:restricting}

To define the auxiliary graphs mentioned above, pick a vertex $w \in G$, and let, for every $i\in\N$, $S_i$ be the set of vertices of $G$ whose graph-theoretical distance from $w$ is at most $i$; also let $S_{-1}=\emptyset$. Note that $S_0 = \{w\}$, every $S_i$ is finite, and $\bigcup_{i \in \N} S_i=V(G)$. For every $i\in\N$, define $\tilde S_i$ to be the subgraph of $G$ on $S_{i+1}$, containing those edges of $G$ that are incident with a vertex in $S_i$. Let $\hat S_i$ be the graph obtained from $\tilde S_i$ by joining every two vertices in $S_{i+1}-S_i$ that lie in the same component of $G-S_i$ with an edge; these new edges are the \emph{outer edges} of $\hat S_i$. For every $i\in\N$, a metric representation $(|\hat S_i|,\ell_i)$ can be defined as follows: let every edge $e$ of $\hat S_i$ that also lies in $\tilde S_i$ have the same length as in $|G|$, and let every outer edge $e=uv$ of $\hat S_i$ have length $d_\ell(u,v)$.
For any two points $x,y\in |\hat S_i|$ we will write $d_i(x,y)$ for $d_{\ell_i}(x,y)$ (the latter was defined at the end of \Sr{subsec:notation:infinite}). 
Recall that in the previous subsection we defined a length $\ell_i(X)$ for every   path, cycle, element of the cycle space, or image of a topological path $X$ in $|\hat S_i|$.

If $X$ is an arc with endpoints in $\hat V$ or a circle in $|G|$, define the \emph{restriction $X | \hat S_i$ of $X$ to $\hat S_i$} as follows. If $X$ avoids $S_i$, let $X | \hat S_i=\emptyset$. Otherwise, start with $E(X) \cap E(\hat S_i)$ and add all outer edges $uv$ of $\hat S_i$ such that $X$ contains a $u$--$v$~arc that meets $\hat S_i$ only in $u$ and $v$. We defined $X | \hat S_i$ to be an edge set, but we will, with a slight abuse, also use the same term to denote the subgraph of $\hat S_i$ spanned by this edge set. Clearly, the restriction of a circle is a cycle and the restriction of an arc is a path. For a path or cycles $X$ in $\hat S_j$ with $j>i$, we define the restriction $X | \hat S_i$ to $\hat S_i$ analogously.

\showFig[0.45\linewidth]{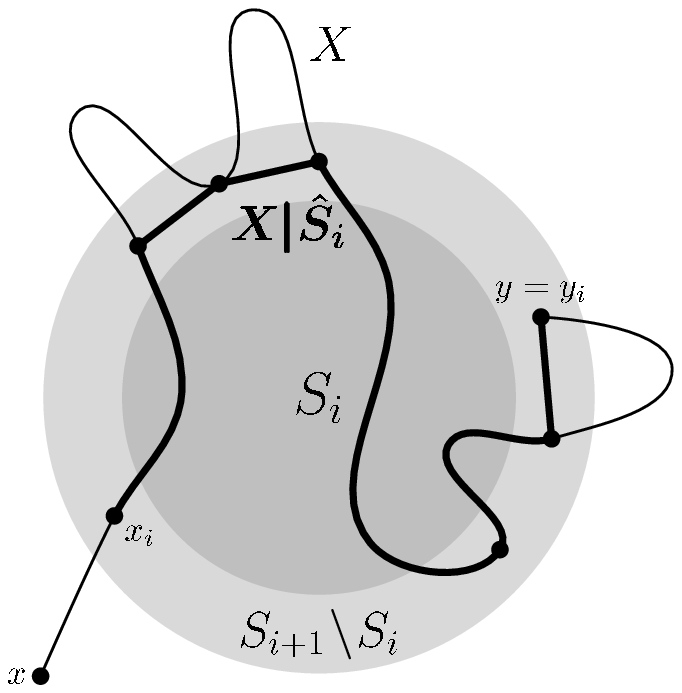}{fig:bsp:restriction}{The restriction of an $x$--$y$~arc $X$ to the $x_i$--$y_i$~path $X | \hat S_i$.}

Note that in order to obtain $X | \hat S_i$ from $X$, we deleted a set of edge-disjoint arcs or paths in $X$, and for each element of this set we put in $X | \hat S_i$ an outer edge with the same endpoints. As no arc or path is shorter than an outer edge with the same endpoints, we easily obtain:

\begin{lemma}
  \label{lemma:shorter}
  Let $i\in\N$ and let $X$ be an arc or a circle in $|G|$ (respectively, a path or cycle in $\hat S_j$ with $j>i$). Then $\ell_i(X | \hat S_i) \leq \ell(X)$ (resp.\ $\ell_i(X | \hat S_i) \leq \ell_j(X)$).
\end{lemma}

A consequence of this is the following:

\begin{lemma}
  \label{lemma:shestPath}
If $x,y \in S_{i+1}$ and $P$ is a shortest $x$--$y$~path in $\hat S_i$ with respect to $\ell_i$ then $\ell_i(P)=d(x,y)$.
\end{lemma}

\begin{proof}
Suppose first that $\ell_i(P)<d(x,y)$. Replacing every outer edge $uv$ in $P$ by a $u$--$v$~arc of length $\ell_i(uv)+\eps$ in $|G|$ for a sufficiently small $\eps$, we obtain a topological $x$--$y$~path in $|G|$ whose image is shorter than $d(x,y)$. Since, by \Lr{arc}, the image of every topological path contains an arc with the same endpoints, this contradicts the definition of $d(x,y)$. Next, suppose that $\ell_i(P)>d(x,y)$. In this case, there is by the definition of $d(x,y)$ an $x$--$y$~arc $Q$ in $|G|$ with $\ell(Q)<\ell_i(P)$. Then $\ell_i(Q|\hat S_i)\leq \ell(Q)<\ell_i(P)$ by \Lr{lemma:shorter}, contradicting the choice of $P$.
\end{proof}

Let $C\in \CCC(G)$. For the proof of \Tr{thm:generating} we will construct a family of $\ell$-geodetic circles in $\omega$ steps, choosing finitely many of these at each step. To ensure that the resulting family will be thin, we will restrict the lengths of those circles: the next two lemmas will help us bound these lengths from above, using the following amounts $\eps_i$ that vanish as $i$ grows.
\begin{equation*}
\eps_i\assign\sup \{d(x,y) \mid x,y \in |G|\emtext{ and there is an } x\emtext{--}y~\emtext{arc in } |G| \setminus G[S_{i-1}]\}.
\end{equation*}
The space $|G| \setminus G[S_{i-1}]$ considered in this definition is the same as the union of $|G-S_{i-1}|$ and the inner points of all edges from $S_{i-1}$ to $V(G) \setminus S_{i-1}$. Note that as $|G|$ is compact, each $\eps_i$ is finite.

\begin{lemma}
  \label{lemma:shortcycles}
  Let $j\in\N$, let $C$ be a cycle in $\hat S_j$, and let $i\in\N$ be the smallest index such that $C$ meets $S_i$. Then $C$ can be written as a sum of $\ell_j$-geodetic cycles in $\hat S_j$ each of which has length at most $5\eps_i$ in $\hat S_j$.
\end{lemma}

\begin{proof}
  We will say that a cycle $D$ in $\hat S_j$ is \emph{a $C$-sector} if there are vertices $x,y$ on $D$ such that one of the $x$--$y$~paths on $D$ has length at most $\eps_i$ and the other, called a \emph{$C$-part} of $D$, is contained in $C$.

  We claim that every $C$-sector $D$ longer than $5\eps_i$ can be written as a sum of cycles shorter than $D$, so that every cycle in this sum either has length at most $5\eps_i$ or is another $C$-sector. Indeed, let $Q$ be a $C$-part of $D$ and let $x,y$ be its endvertices. Every edge $e$ of $Q$ has length at most $2\eps_i$: otherwise the midpoint of $e$ has distance greater than $\eps_i$ from each endvertex of $e$, contradicting the definition of $\eps_i$. As $Q$ is longer than $4\eps_i$, there is a vertex $z$ on $Q$ whose distance, with respect to $\ell_j$, along $Q$ from $x$ is larger than $\eps_i$ but at most $3\eps_i$.
  Then the distance of $z$ from $y$ along $Q$ is also larger than $\eps_i$. By the definition of $\eps_i$ and \Lr{lemma:shestPath}, there is a $z$--$y$~path $P$ in $\hat S_j$ with $\ell_j(P) \leq \eps_i$.

\showFig[0.6\linewidth]{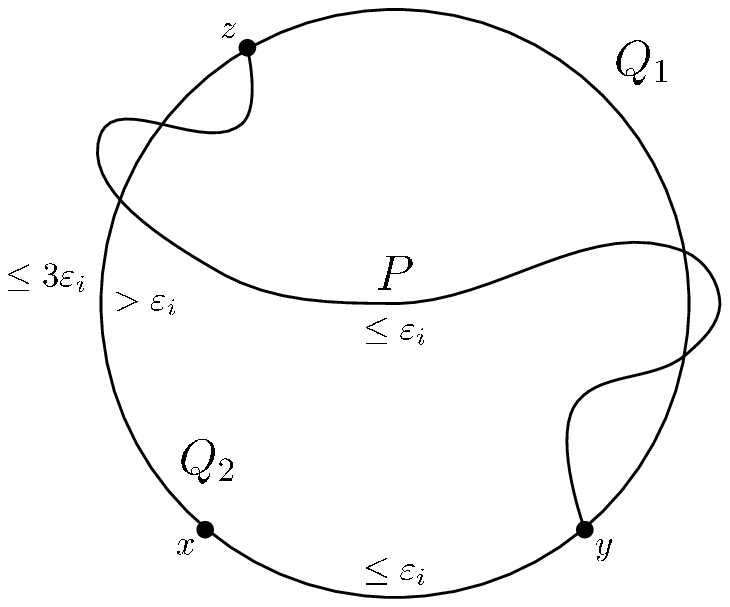}{fig:bsp:Csector}{The paths $Q_1$, $Q_2$, and $P$ in the proof of \Lr{lemma:shortcycles}.}

  Let $Q_1=zQy$ and let $Q_2$ be the other $z$--$y$~path in $D$. (See also Figure~\ref{fig:bsp:Csector}.) Note that $Q_2$ is the concatenation of $zQ_2x$ and $xQ_2y$. Since $\eps_i<\ell_j(zQ_2x)\le 3\eps_i$ and $\ell_j(xQ_2y)\le\eps_i$, we have $\eps_i<\ell_j(Q_2)\le 4\eps_i$. For any two paths $R,L$, we write $R+L$ as a shorthand for the symmetric difference of $E(R)$ and $E(L)$. It is easy to check that every vertex is incident with an even number of edges in $Q_2+ P$, which means that $Q_2+P$ is an element of the cycle space of $\hat S_j$, so by \Lr{pro:cyclespace:circles} it can be written as a sum of edge-disjoint cycles in $\hat S_j$. Since $\ell_j(Q_2+P)\le\ell_j(Q_2)+\ell_j(P)\le 4\eps_i+\eps_i=5\eps_i$, every such cycle has length at most $5\eps_i$. On the other hand, we claim that $Q_1+ P$ can be written as a sum of $C$-sectors that are contained in $Q_1\cup P$. If this is true then each of those $C$-sectors will be shorter than $D$ since

  \begin{equation*}
    \ell_j(Q_1 \cup P)\le\ell_j(Q_1)+\ell_j(P)\le\ell_j(Q_1)+\eps_i<\ell_j(Q_1)+\ell_j(Q_2)=\ell_j(D).
  \end{equation*}

  To prove that $Q_1+P$ is a sum of such $C$-sectors, consider the vertices in $X\assign V(Q_1)\cap V(P)$ in the order they appear on $P$ (recall that $P$ starts at $z$ and ends at $y$) and let $v$ be the last vertex in this order such that $Q_1v+Pv$ is the (possibly trivial) sum of $C$-sectors contained in $Q_1\cup P$ (there is such a vertex since $z\in X$ and $Q_1z+Pz=\emptyset$). Suppose $v\not= y$ and let $w$ be the successor of $v$ in $X$. The paths $vQ_1w$ and $vPw$ have no vertices in common other than $v$ and $w$, hence either they are edge-disjoint or they both consist of the same edge $vw$. In both cases, $Q_1w+Pw=(Q_1v+Pv)+(vQ_1w+vPw)$ is the sum of $C$-sectors contained in $Q_1\cup P$, since $Q_1v+Pv$ is such a sum and $vQ_1w+vPw$ is either the empty edge-set or a $C$-sector contained in $Q_1\cup P$ (recall that $vQ_1w\subset C$ and $\ell_j(vPw)\le\eps_i$). This contradicts the choice of $v$, therefore $v=y$ and $Q_1+P$ is a sum of $C$-sectors as required.

  Thus every $C$-sector longer than $5\eps_i$ is a sum of shorter cycles, either $C$-sectors or cycles shorter than $5\eps_i$. As $\hat S_j$ is finite and $C$ is a $C$-sector itself, repeated application of this fact yields that $C$ is a sum of cycles not longer than $5\eps_i$. By Proposition~\ref{thm:fin:L}, every cycle in this sum is a sum of $\ell_j$-geodetic cycles in $\hat S_j$ not longer than $5\eps_i$; this completes the proof.
\end{proof}

\begin{lemma}
  \label{lemma:shortdistance}
  The sequence $(\eps_i)_{i\in\N}$ converges to zero.
\end{lemma}

\begin{proof}
  The sequence $(\eps_i)_{i\in\N}$ converges since it is decreasing. Suppose there is an $\eps>0$ with $\eps_i>\eps$ for all $i$. Thus, for every $i\in\N$, there is a component $C_i$ of $|G| \setminus G[S_i]$ in which there are two points of distance at least $\eps$. For every $i\in\N$, pick a vertex $c_i\in C_i$. By Corollary~\ref{cor:closure}, there is an end $\omega$ in the closure of the set $\{c_0,c_1,\dotsc\}$ in $|G|$. Let $\hat C(S_i,\oo)$ denote the component of $|G|\setminus G[S_i]$ that contains $\oo$. It is easy to see that the sets $\hat C(S_i,\oo),~i\in\NN$, form a neigbourhood basis of $\oo$ in $|G|$.

  As $U\assign \{x \in |G| \mid d(x,\oo)<\frac12\eps\}$ is open in $|G|$, it has to contain $\hat C(S_i,\omega)$ for some $i$. Furthermore, there is a vertex $c_j\in\hat C(S_i,\omega)$ with $j\ge i$, because $\omega$ lies in the closure of $\{c_0,c_1,\dotsc\}$. As $S_j\supset S_i$, the component $C_j$ of $|G| \setminus G[S_j]$ is contained in $\hat C(S_i,\omega)$ and thus in $U$. But any two points in $U$ have distance less than $\eps$, contradicting the choice of $C_j$.
\end{proof}

This implies in particular that:

\begin{corollary}
  \label{cor:outeredges}
  Let $\eps>0$ be given. There is an $n\in\N$ such that for every $i\ge n$, every outer edge of $\hat S_i$ is shorter than $\eps$.
\end{corollary}

\subsection{Limits of paths and cycles}
\label{subsec:proof:limit}

In this section we develop some tools that will help us obtain $\ell$-geodetic circles as limits of sequences of $\ell_i$-geodetic cycles in the $\hat S_i$.

A \emph{chain} of paths (respectively cycles) is a sequence $X_j,X_{j+1},\dotsc$ of paths (resp.\ cycles), such that every $X_i$ with $i\ge j$ is the restriction of $X_{i+1}$ to $\hat S_i$.

\begin{definition}
  \label{def:limit}
  The \emph{limit} of a chain $X_j,X_{j+1},\dotsc$ of paths or cycles, is the closure in $|G|$ of the set
  \begin{equation*}
    \tilde X\assign\bigcup_{j\le i<\omega}\left(X_i\cap\tilde S_i\right).
  \end{equation*}
\end{definition}

Unfortunately, the limit of a chain of cycles does not have to be a circle, as shown in Figure~\ref{fig:limitnocycle}. However, we are able to prove the following lemma.

\showFig{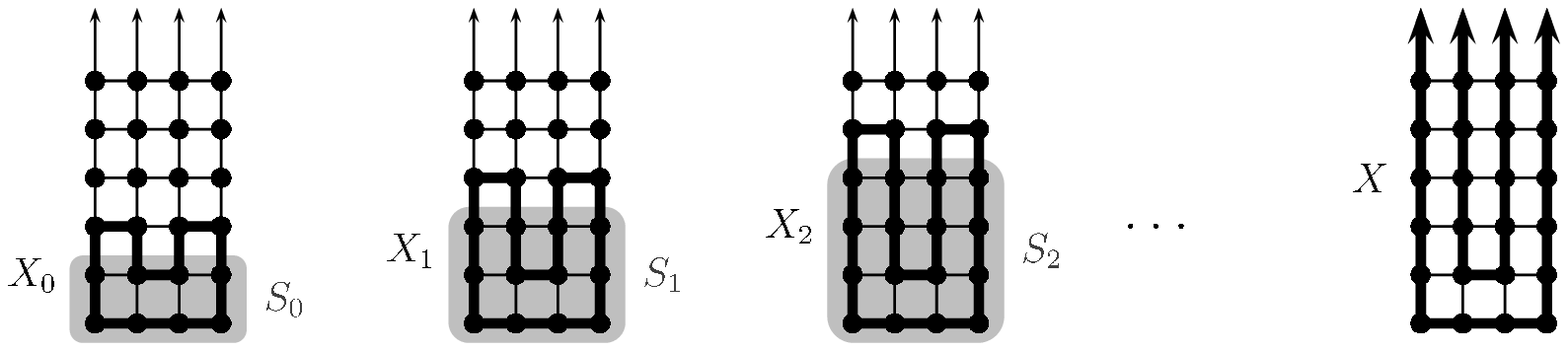}{fig:limitnocycle}{A chain $X_0,X_1,\dotsc$ of cycles (drawn thick), whose limit $X$ is not a circle (but the edge-disjoint union of two circles).}

\begin{lemma}
  \label{lemma:continuous}
  The limit of a chain of cycles is a continuous image of $S^1$ in $|G|$. The limit of a chain of paths is the image of a topological path in $|G|$. The corresponding continuous map can be chosen so that every point in $G$ has at most one preimage, while the preimage of each end of $G$ is a totally disconnected set.
\end{lemma}

\begin{proof}
  Let $X_0,X_1,\dotsc$ be a chain of cycles (proceed analogously for a chain $X_j,X_{j+1},\dotsc$) and let $X$ be its limit. We define the desired map $\sigma:S^1\to X$ with the help of homeomorphisms $\sigma_i:S^1 \to X_i$ for every $i\in\N$. Start with some homeomorphism $\sigma_0:S^1\to X_0$. Now let $i\ge 1$ and suppose that $\sigma_{i-1}:S^1\to X_{i-1}$ has already been defined. We change $\sigma_{i-1}$ to $\sigma_i$ by mapping the preimage of any outer edge in $X_{i-1}$ to the corresponding path in $X_i$. While we do this, we make sure that the preimage of every outer edge in $X_i$ is not longer than $\frac1i$.

  Now for every $x\in S^1$, define $\sigma(x)$ as follows. If there is an $n\in \N$ such that $\sigma_i(x)=\sigma_n(x)$ for every $i\geq n$, then define $\sigma(x)=\sigma_n(x)$. Otherwise, $\sigma_i(x)$ lies on an outer edge $u_iv_i$ for every $i\in\N$. By construction, there is exactly one end $\omega$ in the closure of $\{u_0,v_0,u_1,v_1,\dotsc\}$ in $|G|$, and we put $\sigma(x)\assign\omega$.

  It is straightforward to check that $\sigma:S^1\to X$ is continuous, and that $\tilde{X}\subseteq \sigma(S^1)$. As $\sigma(S^1)$ is a continuous image of the compact space $S^1$ in the Hausdorff space $|G|$, it is closed in $|G|$, thus $\sigma(S^1)=X$. By construction, only ends can have more than one preimage under $\sigma$. Moreover, as we defined $\sigma_i$ so that the preimage of every outer edge is not longer than $\frac1i$, any two points $x,y\in S^1$ that are mapped by $\sigma$ to ends are mapped by $\sigma_i$ to distinct outer edges for every sufficiently large $i$. Therefore, there are points in $S^1$ between $x$ and $y$ that are mapped to a vertex by $\sigma_i$ and hence also by $\sigma$, which shows that the preimage of each end under $\sigma$ is totally disconnected.

  For a chain $X_0,X_1,\dots$ of paths, the construction is slightly different: As the endpoints of the paths $X_i$ may change while $i$ increases, we let $\sigma_i: [0,1]\to X$ map a short interval $[0,\delta_i]$ to the first vertex of $X_i$, and the interval $[1-\delta_i,1]$ to the last vertex of $X_i$, where $\delta_i$ is a sequence of real numbers converging to zero. Except for this difference, the construction of a continuous map $\sigma: [0,1]\to X$ imitates that of the previous case.
\end{proof}

Recall that a circle is $\ell$-geodetic if for every two points $x,y\in C$, one of the two $x$--$y$~arcs in $C$ has length~$d(x,y)$. Equivalently, a circle $C$ is $\ell$-geodetic if it has no \emph{shortcut}, that is, an arc in $|G|$ with endpoints $x,y \in C$ and length less than both $x$--$y$~arcs in $C$. Indeed, if a circle $C$ has no shortcut, it is easily seen to contain a shortest $x$--$y$~arc for any points $x,y\in C$. 

It may seem more natural if a shortcut of $C$ is a $C$\emph{-arc}, that is, an arc that meets $C$ only at its endpoints. The following lemma will allow us to only consider such shortcuts (in particular, we only have to consider arcs with endpoints in $\hat V$).

\begin{lemma}
  \label{lemma:shortcut}
Every shortcut of a circle $C$ in $|G|$ contains a $C$-arc which is also a shortcut of $C$.
\end{lemma}

\begin{proof}
  Let $P$ be a shortcut of $C$ with endpoints $x,y$. As $C$ is closed, every point in $P\setminus C$ is contained in a $C$-arc in $P$. Suppose no $C$-arc in $P$ is a shortcut of $C$.
  We can find a family $(W_i)_{i\in \N}$ of countably many internally disjoint arcs in $P$, such 
  that for every $i$, $W_i$ is either a $C$-arc or an arc contained in $C$, and every edge in $P$ lies in some $W_i$ (there may, however, exist ends in $P$ that are not contained in any arc $W_i$).
  For every $i$, let $x_i,y_i$ be the endpoints of $W_i$ and pick a 
  $x_i$--$y_i$~arc $K_i$ as follows.
  If $W_i$ is contained in $C$, let $K_i = W_i$.
  Otherwise, $W_i$ is a $C$-arc and we let $K_i$ be the shortest 
  $x_i$--$y_i$~arc on $C$.
  Note that since $W_i$ is not a shortcut of $C$, $K_i$ is at most as long as $W_i$.

  Let $K$ be the union of all the arcs $K_i$.
  Clearly, the closure $\overline{K}$ of $K$ in $|G|$ is contained in $C$, 
  contains $x$ and $y$, and is at most as long as $P$.
  It is easy to see that $\overline{K}$ is a connected topological space; indeed, if not, then there are distinct edges $e,f$ on $C$, so that both components of $C - \{e,f\}$ meet $\overline{K}$, which cannot be the case by the construction of $K$.
  By \Tr{pro:connected}, $\overline{K}$ is also arc-connected, and so it 
  contains an $x$--$y$~arc that is at most as long as $P$, contradicting the fact that $P$ is a shortcut of $C$.

  Thus, $P$ contains a $C$-arc which is also a shortcut of $C$.
\end{proof}

By the following lemma, the restriction of any geodetic circle is also geodetic.

\begin{lemma}
  \label{lemma:restriction}
  Let $i>j$ and let $C$ be an $\ell$-geodetic circle in $|G|$ (respectively, an $\ell_i$-geodetic cycle in $\hat S_i$). Then $C_j\assign C | \hat S_j$ is an $\ell_j$-geodetic cycle in $\hat S_j$, unless $C_j=\emptyset$.
\end{lemma}

\begin{proof}
  Suppose for contradiction, that $C_j$ has a shortcut $P$ between the vertices $x,y$. Clearly, $x,y$ lie in $C$, so let $Q_1,Q_2$ be the two $x$--$y$~arcs (resp.\ $x$--$y$~paths) in $C$. We claim that $\ell(Q_k)>d(x,y)$ (resp.\ $\ell_i(Q_k)>d(x,y)$) for $k=1,2$. Indeed, as $P$ is a shortcut of $C_j$, and $Q_k|\hat S_j$ is a subpath of $C_j$ with endvertices $x,y$ for $k=1,2$, we have $\ell_j(Q_k|\hat S_j)>\ell_j(P)$. Moreover, by \Lr{lemma:shorter} we have $\ell(Q_k) \geq \ell_j(Q_k|\hat S_j)$ (resp.\ $\ell_i(Q_k) \geq \ell_j(Q_k|\hat S_j)$), and by \Lr{lemma:shestPath} $\ell_j(P)\geq d(x,y)$, so our claim is proved. But then, by the definition of $d(x,y)$ (resp.\ by \Lr{lemma:shestPath}), there is an $x$--$y$~arc $Q$ in $|G|$ such that $\ell(Q)<\ell(Q_k)$ (resp.\ an $x$--$y$~path $Q$ in $\hat S_i$ such that $d(x,y)=\ell_i(Q)<\ell_i(Q_k)$) for $k=1,2$, contradicting the fact that $C$ is $\ell$-geodetic (resp.\ $\ell_i$-geodetic).
\end{proof}

As already mentioned, the limit of a chain of cycles does not have to be a circle. Fortunately, the limit of a chain of geodetic cycles is always a circle, and in fact an $\ell$-geodetic one:

\begin{lemma}
  \label{lemma:limitgeo}
  Let $C$ be the limit of a chain $C_0,C_1,\dotsc$ of cycles, such that, for every $i\in\NN$, $C_i$ is $\ell_i$-geodetic in $\hat S_i$. Then $C$ is an $\ell$-geodetic circle.
\end{lemma}

\begin{proof}
  Let $\sigma$ be the map provided by \Lr{lemma:continuous} (with $C_i$ instead of $X_i$). We claim that $\sigma$ is injective.

  Indeed, as only ends can have more than one preimage under $\sigma$, suppose, for contradiction, that $\omega$ is an end with two preimages. These preimages subdivide $S^1$ into two components $P_1,P_2$. Choose $\eps \in \R_+$ smaller than the lengths of $\sigma(P_1)$ and $\sigma(P_2)$. (Note that as the preimage of $\omega$ is totally disconnected, both $\sigma(P_1)$ and $\sigma(P_2)$ have to contain edges of $G$, so their lengths are positive.) By Corollary~\ref{cor:outeredges}, there is a $j$ such that every outer edge of $\hat S_j$ is shorter than $\eps$. On the other hand, for a sufficiently large $i\geq j$, the restrictions of $\sigma(P_1)$ and $\sigma(P_2)$ to $\hat S_i$ are also longer than $\eps$. Thus, the distance along $C_i$ between the first and the last vertex of $\sigma(P_1) | \hat S_i$ is larger than $\eps$. As those vertices lie in the same component of $G-S_i$ (namely, in $C(S_i,\omega)$), there is an outer edge of $\hat S_i$ between them. This edge is shorter than $\eps$ and thus a shortcut of $C_i$, contradicting the fact that $C_i$ is $\ell_i$-geodetic.

  Thus, $\sigma$ is injective. As any bijective, continuous map between a compact space and a Hausdorff space is a homeomorphism, $C$ is a circle.

  Suppose, for contradiction, there is a shortcut $P$ of $C$ between points $x,y\in C\cap \hat{V}$. Choose $\eps>0$ such that $P$ is shorter by at least $3\eps$ than both $x$--$y$~arcs on $C$. Then, there is an $i$ such that the restrictions $Q_1,Q_2$ of the $x$--$y$~arcs on $C$ to $\hat S_i$ are longer by at least $2\eps$ than $P_i\assign P | \hat S_i$ ($Q_1,Q_2$ lie in $C_i$ by the definition of $\sigma$, but note that they may have different endpoints). By Corollary~\ref{cor:outeredges}, we may again assume that every outer edge of $\hat S_i$ is shorter than $\eps$. If $x$ does not lie in $\hat S_i$, then the first vertices of $P_i$ and $Q_1$ lie in the component of $G-S_i$ that contains $x$ (or one of its rays if $x$ is an end). The same is true for $y$ and the last vertices of $P_i$ and $Q_1$. Thus, we may extend $P_i$ to a path $P'_i$ with the same endpoints as $Q_1$, by adding to it at most two outer edges of $\hat S_i$. But $P'_i$ is then shorter than both $Q_1$ and $Q_2$, in contradiction to the fact that $C_i$ is $\ell_i$-geodetic. Thus there is no shortcut to $C$ and therefore it is $\ell$-geodetic.
\end{proof}

\subsection{Proof of the generating theorem}
\label{subsec:proof:thm}

Before we are able to prove Theorem~\ref{thm:generating}, we need one last lemma.

\begin{lemma}
  \label{lemma:decompose:shortcycles}
  Let $C$ be a circle in $|G|$ and let $i\in\N$ be minimal such that $C$ meets $S_i$. Then, there exists a finite family $\mathcal F$, each element of which is an $\ell$-geodetic circle in $|G|$ of length at most $5\eps_i$, and such that $\sum \mathcal F$ coincides with $C$ in $\tilde S_i$, that is, $(\sum \mathcal F)\cap \tilde S_i= C\cap \tilde S_i$.
\end{lemma}

\begin{proof}
  For every $j\ge i$, choose, among all families $\mathcal H$ of $\ell_j$-geodetic cycles in $\hat S_j$, the ones that are minimal with the following properties, and let $V_j$ be their set:
  \begin{itemize}
    \item
      no cycle in $\mathcal H$ is longer than $5\eps_i$ in $\hat S_j$, and
    \item
      $\sum \mathcal H$ coincides with $C$ in $\tilde S_i$.
  \end{itemize}
  Note that every cycle in such a family $\mathcal H$ meets $S_i$ as otherwise $\mathcal H$ would not be minimal with the above properties. By \Lr{lemma:shortcycles}, the sets $V_j$ are not empty. As no family in $V_j$ contains a cycle twice, and $\hat S_j$ has only finitely many cycles, every $V_j$ is finite. Our aim is to find a family $\CCC_i\in V_i$ so that each cycle in $\CCC_i$ can be extended to an $\ell$-geodetic circle, giving us the desired family $\mathcal F$.

  Given $j\ge i$ and $\CCC\in V_{j+1}$, restricting every cycle in $\CCC$ to $\hat S_j$ (note that by the minimality of $\CCC$, every cycle in $\CCC$ meets $S_i$ and hence also its supset $S_j$) yields, by \Lr{lemma:restriction}, a family $\CCC^-$ of $\ell_j$-geodetic cycles. Moreover, $\CCC^-$ lies in $V_j$: by \Lr{lemma:shorter}, no element of $\CCC^-$ is longer than $5\eps_i$, and the sum of $\CCC^-$ coincides with $C$ in $\hat S_i$, as the performed restrictions do not affect the edges in $\tilde S_i$. In addition, $\CCC^-$ is minimal with respect to the above properties as $\CCC$ is.

  Now construct an auxiliary graph with vertex set $\bigcup_{j\geq i} V_j$, where for every $j>i$, every element $\CCC$ of $V_j$ is incident with $\CCC^-$. Applying \Lr{lemma:infinity} to this graph, we obtain an infinite sequence $\CCC_i,\CCC_{i+1},\dotsc$ such that for every $j\ge i$, $\CCC_j\in V_j$ and $\CCC_j = \CCC_{j+1}^-$. Therefore, for every cycle $C\in\CCC_j$ there is a unique cycle in $\CCC_{j+1}$ whose restriction to $\hat S_j$ is $C$. Hence for every $D\in\CCC_i$ there is a chain $(D=)D_i,D_{i+1},\dotsc$ of cycles such that $D_j\in\CCC_j$ for every $j\ge i$. By \Lr{lemma:limitgeo}, the limit $X_D$ of this chain is an $\ell$-geodetic circle, and $X_D$ is not longer than $5\eps_i$, because in that case some $D_j$ would also be longer than $5\eps_i$. Thus, the family $\mathcal F$ resulting from $\CCC_i$ after replacing each $D\in\CCC_i$ with $X_D$  has the desired properties.
\end{proof}

\begin{proof}[Proof of Theorem~\ref{thm:generating}]
  If $(\mathcal F_i)_{i \in I}$ is a family of families, then let the family $\famun_{i \in I} \mathcal F_i$ be the disjoint union of the families $\mathcal F_i$. Let $C$ be an element of $\CCC(G)$. For $i=0,1,\dotsc$, we define finite families $\Gamma_i$ of $\ell$-geodetic circles that satisfy the following condition:
  \begin{equation} \label{enum:disjoint}
  \text{$C_i\assign C+\sum \famun_{j\leq i} \Gamma_j$ does not contain edges of $\tilde S_i$,}
  \end{equation}
  where $+$ denotes the symmetric difference.

  By \Lr{pro:cyclespace:circles}, there is a family $\CCC$ of edge-disjoint circles whose sum equals $C$. Applying \Lr{lemma:decompose:shortcycles} to every circle in $\CCC$ that meets $S_0$ (there are only finitely many), yields a finite family $\Gamma_0$ of $\ell$-geodetic circles that satisfies condition~\eqref{enum:disjoint}.

  Now recursively, for $i=0,1,\ldots$, suppose that $\Gamma_0,\ldots,\Gamma_i$ are already defined finite families of $\ell$-geodetic circles satisfying condition~\eqref{enum:disjoint}, and write $C_i$ as a sum of a family $\CCC$ of edge-disjoint circles, supplied by \Lr{pro:cyclespace:circles}. Note that only finitely many members of $\CCC$ meet $S_{i+1}$, and they all avoid $S_i$ as $C_i$ does. Therefore, for every member $D$ of $\CCC$ that meets $S_{i+1}$, \Lr{lemma:decompose:shortcycles} yields a finite family $\mathcal F_D$ of $\ell$-geodetic circles of length at most $5\eps_{i+1}$ such that $(\sum \mathcal F_D)\cap \tilde{S}_{i+1}=D\cap \tilde{S}_{i+1}$. Let \begin{equation*}
  	\Gamma_{i+1} \assign \famun_{\stackrel{D\in \CCC}{D \cap S_{i+1}\neq \emptyset}} \mathcal F_D.
  \end{equation*}
  By the definition of $C_i$ and $\Gamma_{i+1}$, we have
  \begin{equation*}
    C_{i+1}=C+\sum \famun_{j\leq i+1} \Gamma_j = C + \sum \famun_{j\leq i} \Gamma_j + \sum \Gamma_{i+1} =C_i+\sum \Gamma_{i+1}.
  \end{equation*}
  By the definition of $\Gamma_{i+1}$, condition~\eqref{enum:disjoint} is satisfied by $C_{i+1}$ as it is satisfied by $C_i$. Finally, let
  \begin{equation*}
    \Gamma:=\famun_{i<\omega}\Gamma_i.
  \end{equation*}
  Our aim is to prove that $\sum \Gamma=C$, so let us first show that $\Gamma$ is thin.

  We claim that for every edge $e\in E(G)$, there is an $i\in\N$, such that for every $j\ge i$ no circle in $\Gamma_j$ contains $e$. Indeed, there is an $i\in\N$, such that $\eps_j$ is smaller than $\frac15 \ell(e)$ for every $j\ge i$. Thus, by the definition of the families $\Gamma_j$, for every $j\ge i$, every circle in $\Gamma_j$ is shorter than $\ell(e)$, and therefore too short to contain $e$. This proves our claim, which, as every $\Gamma_i$ is finite, implies that $\Gamma$ is thin.

  Thus, $\sum \Gamma$ is well defined; it remains to show that it equals $C$. To this end, let $e$ be any edge of $G$. By \eqref{enum:disjoint} and the claim above, there is an $i$, such that $e$ is contained neither in $C_i$ nor in a circle in $\famun_{j>i} \Gamma_i$. Thus, we have
  \begin{equation*}
    e\notin C_i+\sum \famun_{j>i} \Gamma_j =C+\sum \Gamma.
  \end{equation*}
  As this holds for every edge $e$, we deduce that $C+\sum \Gamma =\emptyset$, so $C$ is the sum of the family $\Gamma$ of $\ell$-geodetic circles.
\end{proof}

\section{Further problems}
\label{sec:further}
It is known that the finite circles (i.e.\ those containing only finitely many edges) of a locally finite graph $G$ generate $\CCC(G)$ (see \cite[Corollary 8.5.9]{diestelBook05}). In the light of this result and \Tr{thm:generating}, it is natural to pose the following question:

\begin{problem}
\label{pr:fin}
Let $G$ be a locally finite graph, and consider a metric representation $(|G|,\ell)$ of $G$. Do the finite $\ell$-geodetic circles generate $\CCC(G)$?
\end{problem}

The answer to Problem~\ref{pr:fin} is negative: Figure~\ref{fig:bsp:infinitecycles} shows a graph with a metric representation where no geodetic circle is finite.

In \Sr{sec:notation} we did not define $d_{\ell}(x,y)$ as the length of a shortest $x$--$y$~arc, because we could not guarantee that such an arc exists. But does it? The following result asserts that it does.

\begin{proposition}
  \label{lemma:shortestarc}
  Let a metric representation $(|G|,\ell)$ of a locally finite graph $G$ be given and write $d=d_{\ell}$. For any two distinct points $x,y\in \hat{V}$, there exists an $x$--$y$~arc in $|G|$ of length $d(x,y)$.
\end{proposition}

\begin{proof}
  Let $\PPP=P_0,P_1,\dotsc$ be a sequence of $x$--$y$~arcs in $|G|$ such that $(\ell(P_j))_{j\in\N}$ converges to $d(x,y)$. Choose a $j\in\N$ such that every arc in $\PPP$ meets $S_j$, where $S_j$, $\tilde S_j$, and $\hat S_j$ are defined as before. Such a $j$ always exists; if for example $x,y \in \OO$, then pick $j$ so that $S_j$ separates a ray in $x$ from a ray in $y$.

  As $\hat S_j$ is finite, there is a path $X_j$ in $\hat S_j$ and a subsequence $\PPP_j$ of $\PPP$ such that $X_j$ is the restriction of any arc in $\PPP_j$ to $\hat S_j$. Similarly, for every $i>j$, we can recursively find a path $X_i$ in $\hat S_i$ and a subsequence $\PPP_i$ of $\PPP_{i-1}$ such that $X_i$ is the restriction of any arc in $\PPP_i$ to $\hat S_i$.

  By construction, $X_j,X_{j+1},\dotsc$ is a chain of paths. The limit $X$ of this chain contains $x$ and $y$ as it is closed, and $\ell(X)\leq d(x,y)$; for if $\ell(X)> d(x,y)$, then there is an $i$ such that $\ell(X_i\cap \tilde S_i)> d(x,y)$, and as $\ell(X_k\cap \tilde S_k)\ge\ell(X_i\cap \tilde S_i)$ for $k>i$, this contradicts the fact that $(\ell(P_j))_{j\in\N}$ converges to $d(x,y)$. By \Lr{lemma:continuous}, $X$ is the image of a topological path and thus, by \Lr{arc}, contains an $x$--$y$~arc $P$. Since $P$ is at most as long as $X$, it has length $d(x,y)$ (thus as $\ell(X)\le d(x,y)$, we have $P=X$.)
\end{proof}

Our next problem raises the question of whether it is possible, given $x,y \in V(G)$, to approximate $d(x,y)$ by finite $x$--$y$~paths:

\begin{problem}
Let $G$ be a locally finite graph, consider a metric representation $(|G|,\ell)$ of $G$ and write $d=d_{\ell}$. Given $x,y \in V(G)$ and $\epsilon \in \R_+$, is it always possible to find a finite $x$--$y$~path $P$ such that $\ell(P) - d(x,y) <\epsilon$?
\end{problem}

Surprisingly, the answer to this problem is also negative. The graph of Figure~\ref{fig:bsp:infinitecycles} with the indicated metric representation is again a counterexample.

\showFig{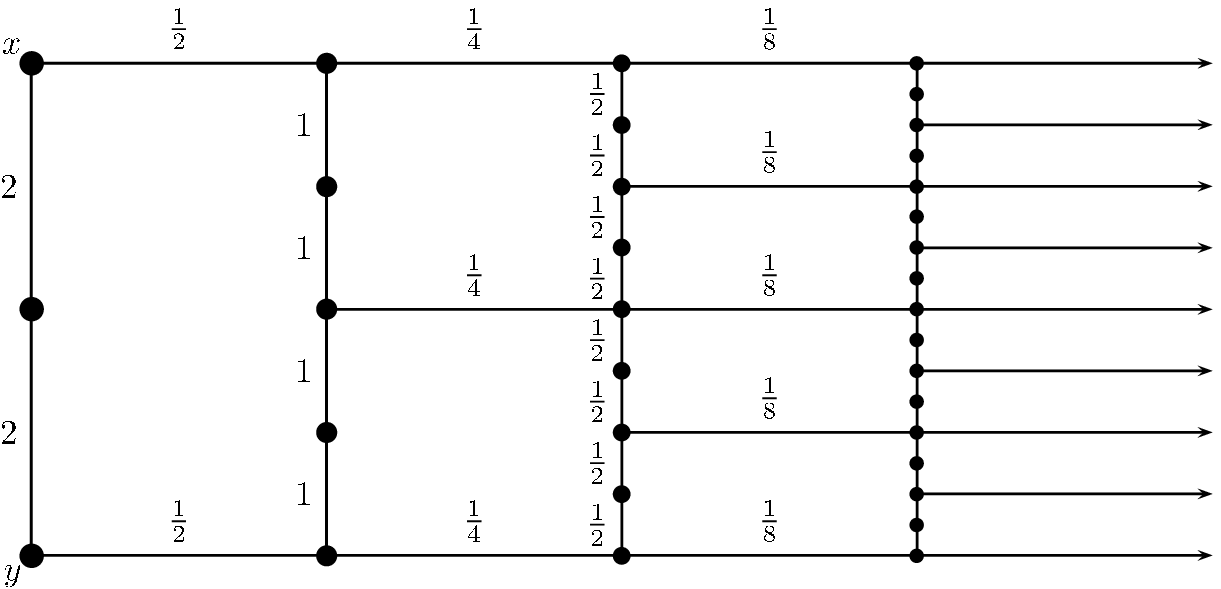}{fig:bsp:infinitecycles}{A $1$-ended graph $G$ with a metric representation. Every $\ell$-geodetic circle is easily seen to contain infinitely many edges. Moreover, every (graph-theoretical) $x$--$y$~path has length at least $4$, although $d(x,y)=2$.}

As noted in \Sr{sec:proof}, every $\ell$-geodetic circle has finite length. But what about other circles? Is it possible to choose a metric representation such that there are circles of infinite length? Yes it is, 
Figure~\ref{fig:bsp:infinitecycles} shows such a metric representation. It is even possible to have every infinite circle have infinite length: 
Let $G$ be the infinite ladder, let the edges of the upper ray have lengths $\frac{1}{2}, \frac{1}{4}, \frac{1}{8}, \ldots$, let the edges of the 
lower ray have lengths $\frac{1}{2}, \frac{1}{3}, \frac{1}{4}, \ldots$, and let the rungs have lengths $\frac{1}{2}, \frac{1}{4}, \frac{1}{8}, \ldots$.
This clearly yields a metric representation, and as any infinite circle contains a tail of the lower ray, it has infinite length. This means that in this example all $\ell$-geodetic circles are finite, contrary to the metric representation in Figure~\ref{fig:bsp:infinitecycles}, where every $\ell$-geodetic circle is infinite.

Theorems~\ref{thm:fin:L} and~\ref{thm:generating} could be applied in order to prove that the cycle space of a graph is generated by certain subsets of its, by choosing an appropriate length function, as indicated by our next problem.  Call a cycle in a finite graph \emph{peripheral}, if it is induced and non-separating.

\begin{problem}
\label{prob:tut}
If $G$ is a $3$-connected finite graph, is there an assignment of lengths $\ell$ to the edges of $G$, such that every $\ell$-geodetic cycle is peripheral?
\end{problem}

We were not able to give an answer to this problem. A positive answer would imply, by \Tr{thm:fin:L}, a classic theorem of Tutte \cite{tutte63}, asserting that the peripheral cycles of a $3$-connected finite graph generate its cycle space. Problem~\ref{prob:tut} can also be posed for infinite graphs, using the infinite counterparts of the concepts involved\footnote{Tutte's theorem has already been extended to locally finite graphs by Bruhn \cite{locFinTutte}}.

\section*{Acknowledgement}
The counterexamples in Figures~\ref{fig:bsp:classic} and~\ref{fig:bsp:infinitecycles} are due to Henning Bruhn. The authors are indebted to him for these counterexamples and other useful ideas.

\bibliographystyle{plain}
\bibliography{collective}

\small
\vskip2mm plus 1fill
\parindent=0pt\obeylines

\bigbreak
Agelos Georgakopoulos {\tt <georgakopoulos@math.uni-hamburg.de>}
Philipp Spr\"ussel {\tt <philipp.spruessel@gmx.de>}
\smallskip
Mathematisches Seminar
Universit\"at Hamburg
Bundesstra\ss e 55
20146 Hamburg
Germany

\end{document}